\documentclass[12pt]{article}
\usepackage{amsmath}
\usepackage{amsthm}
\usepackage{amssymb}

\newtheorem{theorem}{Theorem}
\newtheorem{cor}{Corollary}
\newtheorem{proposition}{Proposition}
\newtheorem{lemma}{Lemma}

\begin{document}

\vspace*{30px}

\begin{center}\Large
\textbf{On Polynomial Identities for Recursive Sequences}
\bigskip\large

Ivica Martinjak and Iva Vrsaljko\\
Faculty of Science\\
University of Zagreb\\
Bijeni\v cka cesta 32, HR-10000 Zagreb\\
Croatia\\


\end{center}


\begin{abstract} 
In this paper we extend the notion of Melham sum to the Pell and Pell-Lucas sequences. While the proofs of general statements rely on the binomial theorem, we prove some spacial cases by the known Pell identities. We also give extensions of obtained expressions to the other recursive sequences.
\end{abstract}


\section{Introduction}

The Pell sequence $(P_n   )_{n \ge 0}$ and the Pell-Lucas sequence $(Q_n )_{n \ge 0}$ are defined as the second order recurrences,
\begin{eqnarray}
&P_{n+2}=2P_{n+1}+P_{n},&  P_0=0, \enspace P_1=1\\
&Q_{n+2}=2Q_{n+1}+Q_{n},&  Q_0=2, \enspace Q_1=2.
\end{eqnarray}
Equivalently, these sequences can be defined as the solutions of Diophantine equations $$x^2-dy^2= \pm 1$$ for $d=2$. More precisely, the pairs $(Q_n/2, P_n)$ are all solutions of these equations. The $n$-th term of the Pell sequence can also be expressed by the closed form equation.
The Pell-Lucas sequence is sometimes called {\it companion Pell sequence} and there is also similar closed form for this sequence. 
 We let $\gamma$ denote the silver ratio, $\gamma:=1+\sqrt{2}$ and we set $\delta:=1-\sqrt{2}$. Then the closed formula for Pell sequence can be written as
\begin{eqnarray} \label{BinetForPell}
P_n&=& \frac{\gamma^n - \delta^n}{\gamma-\delta} 
\end{eqnarray}
while for the companion Pell numbers we have $Q_n=  {\gamma^n - \delta^n} $.
There are many known properties and identities for these sequences \cite{Bravo, Duje, Wloch}. This includes several identities encountering both of the sequences, 
\begin{eqnarray}
Q_n=P_{n-1}+ P_{n+1}
\end{eqnarray}
being the basic one. Recall that the Cassini identity \cite{WeZe} for Pell numbers has form
\begin{eqnarray} \label{CassiniId}
P_{n-1}P_{n+1}- P_n^2 = (-1)^n.
\end{eqnarray}
An elegant proof include the fact that 
			\begin{eqnarray*}\label{Pellmatriceprop}
				\begin{pmatrix}
					0 &1\\
					1 &2
				\end{pmatrix}^ n = 
				\begin{pmatrix}
					P_{n-1} & P_n\\
					P_n & P_{n+1}
				\end{pmatrix},
			\end{eqnarray*}
which can be proved by induction. When applying the Cauchy-Binet theorem for determinants, the statement follows immediately. We will also use relation
\begin{eqnarray} \label{MplusNid}
P_{m+n} = P_{m-1}P_n + P_mP_{n+1},
\end{eqnarray}
for the purpose to prove some polynomial identities for Pell numbers. Identity (\ref{MplusNid}) can be proved by induction.

In what follows, firstly we prove that $(2m+1)n$-th Pell number is represented as a polynomial in $P_n$. Then we extend the notion of Melham sum \cite{Ulat} to the Pell and Pell-Lucas sequences and find related expansions into the power series of $P_n$, where exponents are odd. Finally, we give extensions of the obtained identities for a certain, more general, class of recursive sequences.

\section{The ${(2m+1)n}$-th Pell number as a polynomial in $P_n$}

\begin{proposition} For the Pell sequence $( P_n )_{n \ge 0}$ we have
\begin{eqnarray} \label{idPoli3n}
i) & P_{3n} =& 8P_n^3 + 3 (-1)^n P_n\\ 
ii) & P_{5n} =& 64P_n^5 + 40(-1)^nP_n^3 +5P_n. \label{idPoli5n}
\end{eqnarray}
\end{proposition}

\begin{proof} According to relations (\ref{CassiniId}) and (\ref{MplusNid}) we get
\begin{eqnarray*}
P_{3n} &=& P_{2n+n} = P_{2n-1}P_n+P_{2n}P_{n+1}\\
&=& P_{n-1}^2P_n+P_n^3 + (P_n-1P_n + P_nP_{n+1})(2P_n + P_{n-1})\\
&=& P_n(P_{n-1}^2 + 2P_nP_{n-1} +  P_n^2 + 2P_nP_{n+1} + P_{n-1}^2 + P_n^2 + (-1)^n )\\
&=&  P_n(3P_n^2 + 2(-1)^2 + 2 P_nP_{n+1} + P_{n-1}^2)\\
&=& P_n(3P_n^2 + 2(-1)^2 +  4P_n^2 + P_{n-1}(2P_n+P_{n-1})    )\\
&=& P_n(8P_n^2 + 3(-1)^n).
\end{eqnarray*}
Application of the same relations also proves identities for $P_{5n}$.
\end{proof}
Furthermore, for the next instance when $n$ is odd we have
\begin{eqnarray} \label{idPoli7n}
P_{7n}= 512{P_n}^7-448{P_n}^5+112{P_n}^3-7P_n
\end{eqnarray}
while all coefficients are positive when $n$ is even. 

\begin{theorem} \label{Jennings} 
 For the Pell sequence $( P_n )_{n \ge 0}$ 
\begin{eqnarray} \label{JenningsRelation}
P_{(2m+1)n} = \sum_{i=0}^m (-1)^{n(m+i)} 2^{3i} \frac{2m+1}{2i+1} \binom{m+i}{2i} P_{n}^{2i+1}.
\end{eqnarray} 
\end{theorem}

\begin{proof}
We use equalities (\ref{JenningsLemma1}) and (\ref{JenningsLemma2}), which are results of D. Jennings available in \cite{Jenn} and which can be proved by induction. 

\begin{eqnarray}
\begin{aligned} \label{JenningsLemma1}
\bigg(x^{2m} + \frac{1}{x^{2m}} \bigg) + \bigg(x^{2m-2} + \frac{1}{x^{2m-2}} \bigg)+ \cdots + \bigg(x^{2} + \frac{1}{x^{2}}\bigg) + 1\\
=\sum_{i=0}^m \frac{2m+1}{m+i+1} \binom{m+i+1}{2i+1} \bigg ( x - \frac{1}{x} \bigg)^{2i}
\end{aligned}
\end{eqnarray}

\begin{eqnarray}
\begin{aligned} \label{JenningsLemma2}
\bigg(x^{2m} + \frac{1}{x^{2m}} \bigg) - \bigg(x^{2m-2} + \frac{1}{x^{2m-2}} \bigg)+ \cdots + (-1)^{m+1} \bigg(x^{2} + \frac{1}{x^{2}}\bigg) + (-1)^m\\
=\sum_{i=0}^m (-1)^{m+i} \frac{2m+1}{m+i+1} \binom{m+i+1}{2i+1} \bigg ( x + \frac{1}{x} \bigg)^{2i}
\end{aligned}
\end{eqnarray}

Having in mind Binet formula for the Pell numbers (\ref{BinetForPell}) and the fact that 
\begin{eqnarray} \label{gammaTimesDelta}
\gamma \cdot \delta &=& -1\\ 
\gamma - \delta&=& 2\sqrt{2} \label{gammaPlusDelta}
\end{eqnarray}
we have
\begin{eqnarray} \label{ratioPpnPn}
\frac{P_{pn} }{P_n} = \frac{\gamma^{pn} - \delta^{pn} }{\gamma^{n}-\delta^n} = x^{p-1}+x^{p-2}y+ \cdots + xy^{p-2} + y^{p-1},
\end{eqnarray}
where $x=\gamma^n$ and $y=\delta^{n}=\frac{(-1)^n}{x}$. When $p$ is odd, the r.h.s. of (\ref{ratioPpnPn}) reduces to 
\begin{eqnarray}
\bigg(x^{p-1} + \frac{1}{x^{p-1}} \bigg) + (-1)^n \bigg(x^{p-1} + \frac{1}{x^{p-1}} \bigg) + \cdots + \bigg(x^{p-1} + \frac{1}{x^{p-1}} \bigg)  + (-1)^n \label{caseA}
\end{eqnarray}
when $p \equiv 3 \pmod {4}$ or to
\begin{eqnarray}
\bigg(x^{p-1} + \frac{1}{x^{p-1}} \bigg) + (-1)^n \bigg(x^{p-1} + \frac{1}{x^{p-1}} \bigg) + \cdots + (-1)^n \bigg(x^{p-1} + \frac{1}{x^{p-1}} \bigg)  + 1 \label{caseB}
\end{eqnarray}
 when $p \equiv 1 \pmod {4}$. Now, we have
$$
x + \frac{1}{x} = \gamma^n + \frac{1}{\gamma^n} = \gamma^n + (-1)^n \delta^n
$$
which gives 
\begin{eqnarray}
x+ \frac{1}{x} &=& (\gamma - \delta) P_n, \enspace n \equiv 1 \pmod {2}\\
x- \frac{1}{x} &= &(\gamma - \delta) P_n, \enspace n \equiv 0 \pmod {2}
\end{eqnarray}
and furthermore
\begin{eqnarray}
\bigg( x+ \frac{1}{x} \bigg )^2 &=& 8 P_n^2, \enspace n \equiv 1 \pmod {2}\\
\bigg( x- \frac{1}{x} \bigg)^2 &= & 8 P_n^2, \enspace n \equiv 0 \pmod {2}
\end{eqnarray}
Since we get expression (\ref{caseA}) assuming that $p$ is odd we now substitute $p=2m+1$. Now, when $n$ is even we obtain all positive terms in (\ref{caseA}) and then r.h.s. of (\ref{ratioPpnPn}) is equal to the l.h.s. of equality (\ref{JenningsLemma1}), 
\begin{eqnarray} \label{JenningsRelationFinal}
P_{(2m+1)n} = \sum_{i=0}^m (-1)^{n(m+i)} 2^{3i} \frac{2m+1}{m+ i+1} \binom{m+i +1}{2i+1} P_{n}^{2i+1}.
\end{eqnarray} 
Analogue reasoning when $n$ is odd gives the same relation, thus (\ref{JenningsRelationFinal}) holds true for any natural number $n$. Finally, a simple manipulation with (\ref{JenningsRelationFinal}) leads to the final form of the theorem.
\end{proof}

One can easily see that relations (\ref{idPoli3n}), (\ref{idPoli5n}) and (\ref{idPoli7n}) appear from Theorem \ref{Jennings} for $m=$1,2 and 3, respectively. When $m=4$ Theorem \ref{Jennings} gives
\begin{eqnarray} \label{idPoli9n}
P_{9n}= 2^{12}{P_n}^9  -9 \cdot 2^9{P_n}^7 + 1728{P_n}^5 -240{P_n}^3 + 9P_n
\end{eqnarray}
when $n$ is odd while all coefficients are positive otherwise. Note that the leading coefficient in (\ref{JenningsRelation}) is always a power of 2, $2^{3m}$, while the absolute value of the coefficient in the term of the smallest degree is $2m+1$.

\section{Melham sum for the Pell and Pell-Lucas sequence}

\begin{proposition}
Twice the sum of the Pell numbers having even indexes from 2 to n is equal to the (2n+1)-st Pell number diminished by 1,
\begin{eqnarray} \label{MelhamTrivial}
1+ 2 \sum_{k=1}^n P_{2k}=P_{2n+1}.
\end{eqnarray}
\begin{proof}
The statement follows immediately from defining properties of Pell sequence,
\begin{eqnarray*}
P_{2n+1}&=& 2P_{2n}+P_{2n-1}\\
&=&2P_{2n}+2P_{2n-2}+P_{2n-3}\\
&=&2P_{2n}+2P_{2n-2}+ \cdots + 2P_2+P_1.
\end{eqnarray*}
\end{proof}
\end{proposition}

Note that relation (\ref{MelhamTrivial}) can be seen as the expansion of the expression $Q_1 \sum_{k=1}^nP_{2k}$ into polynomial in $P_{2n+1}$,
\begin{eqnarray*}
Q_1 \sum_{k=1}^n P_{2k}=P_{2n+1} -1.
\end{eqnarray*}
In what follows we extend this idea to full generality. The expression
$$Q_1Q_2 \cdots Q_{2m+1} \sum_{k=1}^n P_{2k}^{2m+1},$$
we shall call the {\it Melham sum for Pell and Pell-Lucas sequences}, because there is analogy with established term for Fibonacci and Lucas sequences. More on the Fibonacci sequence one can find in \cite{Vajda}. Introduction to Fibonacci polynomials one can find in \cite{GKP}, and some recent development in \cite{ACMS}.

\begin{lemma}		\label{MSlema1}
For the sequences  $ (P_n)_{n \geq 0}$, $(Q_n)_{n \geq 0}$ and  $m \in \mathbb{N}$
\begin{eqnarray}
Q_m\sum_{k=1}^nP_{2mk}=P_{m(2n+1)}-P_m. 
\end{eqnarray}
\end{lemma}

\begin{proof}
By relation (\ref{MplusNid}) we have
\begin{eqnarray*}
P_{m+n}&=&(P_{m-n}-P_nP_{n-1}{(-1)}^n){(-1)}^{n+1}+P_mP_{n-1}\\	
&=&P_{m-1}{(-1)}^{n+1}+P_mP_{n-1}+P_mP_{n+1}\\
&=&P_{m-1}{(-1)}^{n+1}+P_m(P_{n-1}+P_{n+1})\\
&=&P_mQ_n - (-1)^nP_{m-n}.
\end{eqnarray*}
Now we prove the statement of lemma by induction where this result is used in a step of induction. Thus, from the fact that the statement holds true for $n=1$ we have to derive equality $Q_n\sum_{k=1}^{n+1}P_{2mk}=P_{m(2n+3)}-P_m$. We have
\begin{eqnarray*}
Q_n\sum_{k=1}^{n+1}P_{2mk}&=&Q_n \Big (\sum_{k=1}^{n}P_{2mk}+P_{2m(n+1)} \Big) \\
&=&P_{m(2n+1)}-P_m+Q_nP_{2m(n+1)}\\
&=&P_{m(2n+1)}+QP_{2m(n+1)}-P_m\\
&=&P_{2m(n+1)+m}-P_m\\
&=&P_{m(2n+3)}-P_m,
\end{eqnarray*}
which completes the statement of lemma.
\end{proof}

\begin{lemma}	\label{MSlema2}
For the sequences  $ (P_n)_{n \geq 0}$, $(Q_n)_{n \geq 0}$ and  $m \in \mathbb{N}$ 
\begin{eqnarray}
{P_n}^{2m+1}=\frac{1}{2^{3m}}\sum_{j=0}^m{(-1)}^{j(n+1)}\binom{2m+1}{j}P_{(2m+1-2j)n}.
\end{eqnarray}
\end{lemma}

\begin{proof}
By means of binomial theorem and using (\ref{gammaTimesDelta}) as well as (\ref{gammaPlusDelta}) we have
\begin{eqnarray*}
{P_n}^{2m+1}&=&{\Big(\frac{\gamma^n-\delta^n}{\gamma-\delta}\Big)}^{2m+1}\\
&=&\frac{1}{{(\gamma-\delta)}^{2m+1}}\sum_{j=0}^{2m+1}{(-1)^{j+1}}\binom{2m+1}{j}\gamma^{jn}\delta^{(2m+1-j)n}\\
&=&\frac{1}{8^m(\gamma-\delta)}\sum_{j=0}^m{(-1)}^j\binom{2m+1}{j}(\gamma^{(2m+1-j)n}\delta^{jn}-\gamma^{jn}\delta^{(2m+1-j)n})\\
&=&\frac{1}{2^{3m}}\sum_{j=0}^m{(-1)}^j\binom{2m+1}{j}\gamma^{jn}\delta^{jn}(\frac{\gamma^{(2m+1-2j)n}-\delta^{(2m+1-2j)n}}{\gamma-\delta})\\
&=&\frac{1}{2^{3m}}\sum_{j=0}^m{(-1)}^{j(n+m)}P_{(2m+1-2j)n}
\end{eqnarray*}
which completes the statement of lemma.
\end{proof}

\begin{theorem} \label{MS1}
For $m \in \mathbb{N}$ and the sequences $(P_n)_{n \geq 0}$, $(Q_n)_{n \geq 0}$ 
\begin{eqnarray}
\sum_{k=1}^n{P_{2k}}^{2m+1}=\frac{1}{2^{3m}}\sum_{j=0}^{m}  \frac{ (-1)^j }{Q_{2m+1-2j}}  \binom{2m+1}{j}(P_{(2m+1-2j)(2n+1)}-P_{2m+1-2j}).	\label{MSum1}
\end{eqnarray}
\end{theorem}

\begin{proof}
In Lemma \ref{MSlema2} we substitute $n=2k$ and then sum both sides of equality from $k=1$ through $n$. It follows
\begin{eqnarray*}
{P_{2k}}^{2m+1}=\frac{1}{2^{3m}}\sum_{j=0}^m{(-1)}^{j}\binom{2m+1}{j}\sum_{k=1}^nP_{(2m+1-2j)2k}.
\end{eqnarray*} 
When we substitute $\sum_{k=1}^nP_{(2m+1-2j)2k}$ by the expression in Lemma \ref{MSlema1}, the proof is completed.
\end{proof}

\begin{theorem}  \label{Melham}
For $m \in \mathbb{N}$ and sequences $(P_n)_{n \geq 0}$, $(Q_n)_{n \geq 0}$ 
\begin{eqnarray}
\begin{split}
\sum_{k=1}^n P_{2k}^{2m+1} = \sum_{i=0}^m P_{2n+1}^{2i+1} \sum_{j=0}^{m-i} \frac{ (-1)^{m+i} 2^{3(i-m)} (2m-2j+1)  }{Q_{2m+1-2j} (2i+1) } \binom{2m+1}{j} \binom{m-j+i}{2i}\\ + 												\sum_{j=0}^m \frac{ (-1)^{j+1}  P_{2m+1-2j}  }{ 2^{3m} Q_{2m+1-2j}  } \binom{2m+1}{j}. \label{MelhamId}
\end{split}
\end{eqnarray}
\end{theorem}

\begin{proof}
When substitute $m$ with $m-j$ and $n$ with $2n+1$ in Theorem \ref{Jennings} one get
\begin{eqnarray*}
P_{(2m+1-2j)(2n+1)}=\sum_{i=0}^{m-j} (-1)^{(2n+1)(m-j+i)} 2^{3i} \frac{2m-2j+1}{2i+1} \binom{m-j+i}{2i} P_{2n+1}^{2i+1}.
\end{eqnarray*}
We substitute this expression in Theorem \ref{MS1} and the statement follows immediately.
\end{proof}

Now we consider some particular cases of Theorem \ref{Melham}. When $m=1$ we obtain
\begin{eqnarray*}
\sum_{k=1}^n P_{2k}^3 = \frac{1}{14} \big (  P_{2n+1} ^3 - 3 P_{2n+1} +2 \big ).
\end{eqnarray*}
When multiply this relation with $Q_1Q_3$ we get polynomial identity for the Melham sum in case $m=1$
\begin{eqnarray}
Q_1Q_3 \sum_{k=1}^n P_{2k}^3 =2 P_{2n+1} ^3 - 6 P_{2n+1} +4.
\end{eqnarray}
The next case, when $m=2$ gives
\begin{eqnarray}
Q_1Q_3 Q_5 \sum_{k=1}^n P_{2k}^5 = 28 P_{2n+1} ^5 - 120 P_{2n+1}^ 3 + 220 P_{2n+1} -128.
\end{eqnarray}

\section{Further extensions}

Given $s,t \in \mathbb{N}$ and $n \in \mathbb{N}_0$ we define the second order recurrence with the relation
\begin{eqnarray} \label{secondOrderRec}
a_{n+2} = s a_{n-1} + t a_{n}
\end{eqnarray}
and initial values $a_0$ and $a_1$. We say that a sequence $(a_n)_{n \ge 0}$ is a solution of (\ref{secondOrderRec}) if its terms satisfies this recurrence. Here we consider a class of (\ref{secondOrderRec}) defined by $t=1$ and initial terms $a_0 =0$, $a_1=1$. We let $(A_n)_{n \ge 0}$ denote the sequence defined by this class. It is worth mentioning that two notable representatives of this class are Fibonacci and Pell numbers.

\begin{proposition} \label{2nplus1A}
For the sequence of numbers $(A_n)_{n \ge 0}$ we have
\begin{eqnarray}
i) & A_{3n} =& (s^2 + 4) A_n^{3} + 3(-1)^{n} A_n\\
ii) & A_{5n} =& (s^2 + 4)^2 A_n^{5} + 5(s^2+4) (-1)^{n} A_n^{3} + 5A_n.
\end{eqnarray}
\end{proposition}
\begin{proof}
By induction we prove that 
\begin{eqnarray} \label{CassiniA}
A_{n-1} A_{n+1} - A^2 = (-1)^2
\end{eqnarray}
and also 
\begin{eqnarray} \label{MplusNA}
A_{m+n} = A_{m-1}A_n + A_{m}A_{n+1}.
\end{eqnarray}

Now we employ (\ref{MplusNA}) to get 
\begin{eqnarray*}
A_{3n}&=& A_{2n+n} = A_{2n-1} A_n + A_{2n} A_{n+1}\\
 &=& A_{n-1}^2 A_n + A_n^3 + (A_{n-1}A_n  + A_nA_{n+1}) (A_2A_n + A_{n-1} ) \\
&=& A_n(A_{n-1}^2  + A_n^2 + s A_{n-1}A_n + s A_n A_{n+1} + A_{n-1}^2  + A_{n-1}A_{n+1}  ).
\end{eqnarray*}
Having in mind that 
\begin{eqnarray*}
 A_{n-1}^2 + sA_{n-1}A_n = A_n^2 + (-1)^n
\end{eqnarray*}
by (\ref{CassiniA}), we obtain
\begin{eqnarray*}
 A_{3n} = A_n( 2A_n^2 + 2(-1) ^n + A_n^2 + s A_n A_{n+1} + A_{n-1}^2 ).
\end{eqnarray*}
When applying again (\ref{CassiniA}) to the terms $s A_n A_{n+1}$ and $A_{n-1}^2$ we finally have
\begin{eqnarray*}
 A_{3n} &=& A_n[ 4A_n^2 + s^2 A_n^2  + 3(-1)^n ] \\
  &=& A_n \big [(s^2 + 4)A_n^2 + 3 (-1)^n \big ].
\end{eqnarray*}
The second relation can be proved by analogue calculation.
\end{proof}

Clearly, further identities can be proved in the same fashion as Proposition \ref{2nplus1A} is proved. Instead, we give a more elegant family of identities (\ref{JenningsA}) that generalize Proposition \ref{2nplus1A}. It follows as a corollary of the Theorem \ref{Jennings}.
\begin{cor} \label{corJennings}
For $m \in \mathbb{N}$ and the sequence of numbers $(A_n)_{n \ge 0}$ we have
\begin{eqnarray} \label{JenningsA} 
A_{(2m+1)n} = \sum_{i=0}^m (-1)^{n(m+i)} (s^2 + 4)^{i} \frac{2m+1}{2i+1} \binom{m+i}{2i} A_{n}^{2i+1}.
\end{eqnarray} 
\end{cor}

In order to prove Corollary \ref{corJennings} we use the fact that the closed form relation for the terms of sequence $(A_n)_{n \ge 0}$ is
\begin{eqnarray}
A_n= \frac{\alpha^n - \beta^n}{ \alpha - \beta},
\end{eqnarray}
where 
\begin{eqnarray*}
\alpha &=&  \frac{1}{2} (s + \sqrt{s^2 + 4})\\
\beta &=& \frac{1}{2} (s - \sqrt{s^2 + 4}).
\end{eqnarray*}
Furthermore, it holds
\begin{eqnarray*}
\alpha \cdot \beta &=&-1\\
\alpha - \beta &=& \sqrt{s^2 + 4},
\end{eqnarray*}
what generalize relations (\ref{gammaTimesDelta}) and (\ref{gammaPlusDelta}) in the proof of Theorem \ref{Jennings}. This completes the statement of the Corollary \ref{corJennings}.
 
Further generalizations and extensions of expressions presented in this work are also possible.


\begin{thebibliography}{6}

\bibitem{ACMS} T. Amdeberhan, X. Chen, V. H. Moll, B. E. Sagan, Generalized Fibonacci polynomials and Fibonomial coefficients, Annals of Combinatorics, 18 (2014) 541-562.


\bibitem{Bravo} Jhon J. Bravo et al, Powers in products of terms of Pell's and Pell-Lucas Sequences,  International Journal of Number Theory, 11 (2015), 1259-1274.

\bibitem{Duje} Andrej Dujella, A problem of Diophantus and Pell numbers, Application of Fibonacci Numbers, Vol. 7 (G. E. Bergum, A. N. Philippou, A. F. Horadam, eds.), Kluwer, Dordrecht, (1998), 61-68. 

\bibitem{GKP} R. L. Graham, D. E. Knuth, O. Patashnik, Concrete Mathematics, Addison-Wesley, 1994.


\bibitem{Jenn} Derek Jennings, Some polynomial identities for the Fibonacci and Lucas numbers, Fibonacci Quart. 31(2) (1993), 134-137.

\bibitem{Ulat} E. Kilic, N. Omur, Y.T. Ulutas, Alternating sums of the powers of Fibonacci and Lucas numbers, Miskolc Mathematical Notes, 12 (1)  (2011), 87-103.



\bibitem{Vajda} Steven Vajda, Fibonacci and Lucas Numbers, and the Golden Section: Theory and Applications, John Wiley and Sons, Inc., New York, 1989.

\bibitem{WeZe} M. Werman, D. Zeilberger, A bijective proof of Cassini's Fibonacci identity, Discrete Mathematics, 58 (1986), 109.

\bibitem{Wloch} A. Wloch, M. Wolowiec-Musial, Generalized Pell numbers and some relations with Fibonacci numbers, Ars Combinatoria, 109 (2013), 391-403.



\end{thebibliography}
\end{document}